\documentclass[a4paper, 12pt]{amsart}
\usepackage{hyperref,upref,amssymb,tikz,amssymb, amstext, amscd, amsmath, amsthm}

\numberwithin{equation}{section}

\newtheorem{thm}{Theorem}[section]
\newtheorem{lemma}[thm]{Lemma}
\newtheorem{prop}[thm]{Proposition}
\newtheorem{cor}[thm]{Corollary}
\newtheorem{conj}[thm]{Conjecture}
\theoremstyle{definition}
\newtheorem{defn}[thm]{Definition}
\newtheorem{ntn}[thm]{Notation}
\theoremstyle{remark}

\newtheorem{rmk}[thm]{Remark}

\newcommand{\bN}{{\mathbb{N}}}

\renewcommand{\O}{{\mathcal{O}}}

\newcommand{\fs}{{\mathfrak{s}}}
\newcommand{\ft}{{\mathfrak{t}}}

\newcommand{\upchi}{{\raise.35ex\hbox{\ensuremath{\chi}}}}

\newcommand{\Aut}{\operatorname{Aut}}

\newcommand{\id}{{\operatorname{id}}}

\newcommand{\spn}{\operatorname{span}}

\newcommand{\ca}{\mathrm{C}^*}

\newcommand{\mt}{\varnothing}

\title[Remarks On the K-theory of C*-Algebras of Products of Odometers]{Remarks On the K-theory of C*-Algebras of Products of Odometers}

\author{Hui Li}

\address{Hui Li, Department of Mathematics and Physics, North China Electric Power University, Beijing 102206, China}
\email{lihui8605@hotmail.com}

\thanks{The author was supported by National Natural Science Foundation of China (Grant No.~11801176).}

\begin{document}

\begin{abstract}
we pose a conjecture on the K-theory of the self-similar $k$-graph C*-algebra of a standard product of odometers. We generalize the C*-algebra $\mathcal{Q}_S$ to any subset of $\mathbb{N}^\times \setminus \{1\}$ and then realize it as the self-similar $k$-graph C*-algebra of a standard product of odometers.
\end{abstract}

\maketitle

\section{Introduction}

Motivated by the work of Bost and Connes in \cite{BostConnes:SM(95}, Cuntz in \cite{Cun08} constructed a C*-algebra $\mathcal{Q}_{\mathbb{N}}$ which is strongly related to the $ax+b$-semigroup over $\mathbb{N}$. Later Li in \cite{Li12} defined the notion of semigroup C*-algebras and Brownlowe, Ramagge, Robertson, and Whittaker in \cite{BRRW14} defined the boundary quotients of semigroup C*-algebras. Their work branches out to many interesting mathematical areas, hence generates a very popular area of C*-algebras. 

For any nonempty subset $S \subset \mathbb{N}^\times \setminus \{1\}$ consisting of mutually coprime numbers, Barlak, Omland, and Stammeier in \cite{MR3784245} defined a C*-algebra $\mathcal{Q}_S$ which is a direct generalization of the Cuntz algebra $\mathcal{Q}_{\mathbb{N}}$ (by letting $S$ be the set of all prime numbers). Barlak, Omland, and Stammeier decomposed the $K$-theory of $\mathcal{Q}_S$ into a free abelian part which they solved in \cite{MR3784245} and a highly nontrivial torsion part. At the end of \cite{MR3784245} Barlak, Omland, and Stammeier made a conjecture to the torsion part which is equivalent to a conjecture of $k$-graph C*-algebras about whether the K-theory of the single-vertex $k$-graph C*-algebra is independent of the factorization rule. Very recently, many authors also found the $k$-graph algebra conjecture is highly connected with the famous Yang-Baxter equation (see for example \cite{FGS20, MR3454414}). Therefore, the conjecture of Barlak, Omland, and Stammeier and the conjecture about the $k$-graph C*-algebra are extremely important in many ways.

Self-similar actions appear naturally in geometric group theory. Nekrashevych was the first who systematically built a bridge between the C*-algebra theory and the self-similar actions (on rooted trees, see \cite{MR2119267, Nek09}). Nekrashevych's work was recently generalized to the self-similar directed graphs, self-similar directed graph C*-algebras by Exel and Pardo in \cite{EP17}, and to the self-similar $k$-graphs, self-similar $k$-graph C*-algebras by Li and Yang in \cite{LY19_3}. 

The purpose of this paper is two fold. Firstly we pose a conjecture (see Conjecture~\ref{decompose K(O_G,Lambda)}) on the K-theory of the self-similar $k$-graph C*-algebra of a standard product of odometers (see Definition~\ref{standard prod of odo}). Secondly we generalize the C*-algebra $\mathcal{Q}_S$ to any subset of $\mathbb{N}^\times \setminus \{1\}$ and then realize it as the self-similar $k$-graph C*-algebra of a standard product of odometers. Therefore Conjecture~\ref{decompose K(O_G,Lambda)} naturally includes the conjecture of Barlak, Omland, and Stammeier and the conjecture about the $k$-graph C*-algebra. We hope the self-similar $k$-graph C*-algebra setting will provide insight on how to solve these conjectures in the future.

Our paper is organized as follows. In Section~2, we provide the background material about self-similar $k$-graph C*-algebras. In Section~3, we use the skew product approach to stabilize every self-similar $k$-graph C*-algebra. In Section~4, we pose a conjecture (see Conjecture~\ref{decompose K(O_G,Lambda)}) on the K-theory of the self-similar $k$-graph C*-algebra of a standard product of odometers and then we discuss the relationship between Conjecture~\ref{decompose K(O_G,Lambda)} and the conjecture of Barlak, Omland, and Stammeier.

\section{Preliminaries}

In this section, we recap the background of $k$-graphs, $k$-graph C*-algebras, self-similar $k$-graph, self-similar $k$-graph C*-algebras from \cite{KP00, LY19_3, LY19_4}.

\begin{defn}
Let $k$ be a positive integer which is allowed to be infinity. A countable small category $\Lambda$ is called a \emph{$k$-graph} if there exists a functor $d:\Lambda \to \mathbb{N}^k$ satisfying that for $\mu\in\Lambda, p,q \in \mathbb{N}^k$ with $d(\mu)=p+q$, there exist unique $\alpha\in d^{-1}(p)$ and $\beta\in d^{-1}(q) $ with $s(\alpha)=r(\beta)$ such that $\mu=\alpha\beta$.
\end{defn}

\begin{ntn}
Let $\Lambda$ be a $k$-graph. For $p \in \mathbb{N}^k$, denote by $\Lambda^p:=d^{-1}(p)$. For $A,B \subset \Lambda$, denote by $AB:=\{\mu\nu:\mu \in A,\nu\in B,s(\mu)=r(\nu)\}$. Denote by $\{e_i\}_{i=1}^{k}$ the standard basis of $\mathbb{N}^k$. For $1 \leq n \leq k$, denote by $\mathsf{1}_n:=\sum_{i=1}^{n}e_i$. For $p \in \mathbb{N}^k,\lambda \in \mathbb{T}^k$, denote by $\lambda^p:=\prod_{i=1}^{k}\lambda_i^{p_i}$.
\end{ntn}
\begin{defn}
Let $\Lambda$ be a $k$-graph. Then $\Lambda$ is said to be \emph{row-finite} if $\vert v\Lambda^{p}\vert<\infty$ for all $v \in \Lambda^0$ and $p \in \mathbb{N}^k$. $\Lambda$ is said to be \emph{source-free} if $v\Lambda^{p} \neq \mt$ for all $v \in \Lambda^0$ and $p \in \mathbb{N}^k$. $\Lambda$ is said to be \emph{finite} if $|\Lambda^n|<\infty$ for all $n\in \bN^k$.
\end{defn}

\textsf{Throughout the rest of this paper, all $k$-graphs are assumed to be row-finite and source-free.}

\begin{defn}
Let $\Lambda$ be a $k$-graph. Then the \emph{$k$-graph C*-algebra} $\O_\Lambda$ is defined to be the universal C*-algebra generated by a family of partial isometries $\{s_\lambda:\lambda\in\Lambda\}$ (Cuntz-Krieger $\Lambda$-family) satisfying
\begin{enumerate}
\item $\{s_v\}_{v \in \Lambda^0}$ is a family of mutually orthogonal projections;
\item $s_{\mu\nu}=s_{\mu} s_{\nu}$ if $s(\mu)=r(\nu)$;
\item $s_{\mu}^* s_{\mu}=s_{s(\mu)}$ for all $\mu \in \Lambda$; and
\item $s_v=\sum_{\mu \in v \Lambda^{p}}s_\mu s_\mu^*$ for all $v \in \Lambda^0, p \in \mathbb{N}^k$.
\end{enumerate}
\end{defn}

\begin{defn}
Let $G$ be a countable discrete group acting on a $k$-graph $\Lambda$, and let $\vert:G\times \Lambda \to G$ be a map. Then $(G,\Lambda)$ is called a \emph{self-similar $k$-graph} if
\begin{enumerate}
\item $G \cdot \Lambda^p \subset \Lambda^p$ for all $p \in \mathbb{N}^k$;
\item $s(g \cdot \mu)=g \cdot s(\mu)$ and $r(g \cdot \mu)=g \cdot r(\mu)$ for all $g \in G,\mu \in \Lambda$;
\item
$g\cdot (\mu\nu)=(g \cdot \mu)(g \vert_\mu \cdot \nu)$ for all $g \in G,\mu,\nu \in \Lambda$ with $s(\mu)=r(\nu)$;

\item
$g \vert_v =g$ for all $g \in G,v \in \Lambda^0$;

\item
$g \vert_{\mu\nu}=g \vert_\mu \vert_\nu$ for all $g \in G,\mu,\nu \in \Lambda$ with $s(\mu)=r(\nu)$;

\item
$1_G \vert_{\mu}=1_G$ for all $\mu \in \Lambda$;

\item
$(gh)\vert_\mu=g \vert_{h \cdot \mu} h \vert_\mu$ for all $g,h \in G,\mu \in \Lambda$.
\end{enumerate}
\end{defn}

\begin{defn}
\label{D:ssfamily}
Let $(G,\Lambda)$ be a self-similar $k$-graph. Define $\mathcal{O}_{G,\Lambda}^\dagger$ to be the universal unital C*-algebra generated by a Cuntz-Krieger $\Lambda$-family $\{s_\mu:\mu\in\Lambda\}$ and a family of unitaries $\{u_g:g \in G\}$ satisfying
\begin{enumerate}
\item $u_{gh}=u_g u_h$ for all $g,h \in G$;
\item\label{u_g s_mu} $u_g s_\mu=s_{g \cdot \mu} u_{g \vert_\mu}$ for all $g \in G,\mu \in \Lambda$.
\end{enumerate}
Define $\mathcal{O}_{G,\Lambda}:=\overline{\spn}\{s_\mu u_g s_\nu^*:s(\mu)=g \cdot s(\nu)\}$, which is called the \emph{self-similar $k$-graph C*-algebra} of $(G,\Lambda)$.
\end{defn}

\begin{rmk}
There exists a strongly continuous homomorphism $\gamma:\mathbb{T}^k \to \Aut(\mathcal{O}_{G,\Lambda} )$, which is called the \emph{gauge action}, such that $\gamma_\lambda(s_\mu u_g s_\nu^*):=\lambda^{d(\mu)-d(\nu)}s_\mu u_g s_\nu^*$ for all $\mu,\nu \in \Lambda,g\in G$ with $s(\mu)=g \cdot s(\nu)$. By the gauge-invariant uniqueness theorem $\mathcal{O}_{\Lambda}$ embeds in $\mathcal{O}_{G,\Lambda}$ naturally.
\end{rmk}

\begin{rmk}
If $\Lambda$ is finite, then $\mathcal{O}_{G,\Lambda}$ is a unital C*-algebra with the unit $\sum_{v \in \Lambda^0}s_v$.
\end{rmk}

\begin{defn}[{\cite[Definition~4.6]{LY19}}]\label{standard prod of odo}
Let $(G,\Lambda)$ be a self-similar $k$-graph. Then $(G,\Lambda)$ is called a \emph{product of odometers} if
\begin{enumerate}
\item $G=\mathbb{Z}$;
\item $\Lambda^0=\{v\}$;
\item $\Lambda^{e_i}:=\{x_{\fs}^i\}_{\fs =0}^{n_i-1},n_i>1$ for all $1\le i\le k$;
\item $1\cdot {x_\fs^i}=x_{(\fs+1)\ \text{mod } n_i}^i$ for all $1\le i\le k,0\le \fs\le n_i-1$;
\item $1|_{x_\fs^i}= \begin{cases}
    0 &\text{ if }0\le \fs< n_i-1\\
    1 &\text{ if }\fs=n_i-1
\end{cases}
\text{ for all } 1\le i\le k,0\le \fs\le n_i-1$.
\end{enumerate}
Moreover, if $x^i_\fs x^j_\ft = x^j_{\ft'} x^i_{\fs'}$ for all $1\leq i<j \leq k,0\le \fs,\fs'\le n_i-1,0\le \ft,\ft'\le n_j-1$ with $\fs+\ft n_i=\ft'+\fs' n_j$, then $(G,\Lambda)$ is called a \emph{standard product of odometers}. Denote by $g \Lambda:=gcd\{n_i-1:1 \leq i \leq k\}$.
\end{defn}

\section{Self-Similar Skew Products}

Kumjian and Pask in \cite{KP00} defined the notion of skew products of $k$-graphs. That is, given a $k$-graph $\Lambda$ and a functor from $\Lambda$ into a group $K$, they endowed the Cartesian product $\Lambda \times K$ with a $k$-graph structure and it is called the skew product of $\Lambda$ and $K$. Since any $k$-graph $\Lambda$ carries a functor which is the degree map $d:\Lambda \to \mathbb{Z}^k$, there is a natural skew product $\Lambda$ and $\mathbb{Z}^k$, denoted by $\Lambda \star \mathbb{Z}^k$, induced from the degree map (see Definition~\ref{define skew product}).

In this section, for any self-similar $k$-graph $(G,\Lambda)$, we construct a self-similar $k$-graph $(G,\Lambda \star \mathbb{Z}^k)$ and we show that $\mathcal{O}_{G,\Lambda \star \mathbb{Z}^k} \cong \mathcal{O}_{G,\Lambda} \rtimes_\gamma \mathbb{T}^k$, where $\gamma:\mathbb{T}^k \to \Aut(\mathcal{O}_{G,\Lambda} )$ is the gauge action.

\begin{defn}[{\cite[Definition~5.1]{KP00}}]\label{define skew product}
Let $\Lambda$ be a $k$-graph. Define $\Lambda \star \mathbb{Z}^k:=\Lambda \times \mathbb{Z}^k$; define $(\Lambda \star \mathbb{Z}^k)^0:=\Lambda^0 \times \mathbb{Z}^k$; for $(\mu,z) \in \Lambda \star \mathbb{Z}^k$, define $s(\mu,z):=(s(\mu),d(\mu)+z), r(\mu,z):=(r(\mu),z)$; for $(\mu,z),(\nu,d(\mu)+z) \in \Lambda \star \mathbb{Z}^k$ with $s(\mu)=r(\nu)$, define $(\mu,z) \cdot (\nu,d(\mu)+z):=(\mu\nu,z)$; for $(\mu,z) \in \Lambda \star \mathbb{Z}^k$, define $d(\mu,z):=d(\mu)$. Then $\Lambda \star \mathbb{Z}^k$ is a $k$-graph.
\end{defn}

\begin{defn}
Let $(G,\Lambda)$ be a self-similar $k$-graph. For $g \in G,(\mu,z) \in \Lambda \star \mathbb{Z}^k$, define $g \cdot(\mu,z):=(g \cdot \mu,z)$ and $g \vert_{(\mu,z)}:=g \vert_\mu$. Then $(G,\Lambda \star \mathbb{Z}^k)$ is a self-similar $k$-graph. 
\end{defn}

\begin{ntn}
Let $(G,\Lambda)$ be a self-similar $k$-graph. Denote by $\{s_\mu,u_g\}$ the generators of $\mathcal{O}_{G,\Lambda}^\dagger$, denote by $\{t_{(\mu,z)},v_g\}$ the generators of $\mathcal{O}_{G,\Lambda \star \mathbb{Z}^k}^\dagger$. We have $\mathcal{O}_{G,\Lambda \star \mathbb{Z}^k}=\overline{\spn}\{t_{(\mu,z-d(\mu))}v_g t_{(\nu,z-d(\nu))}^*:\mu,\nu \in \Lambda,g \in G, z \in \mathbb{Z}^k,s(\mu)=g \cdot s(\nu)\}$.
\end{ntn}

\begin{lemma}\label{operators in multiplier}
Let $A$ be a C*-algebra, let $S$ be a set of generators of $A$, and let $(T_i)_{i \in I}$ be a net of operators in $M(A)$ such that $(T_i)_{i \in I}$ is uniformly bounded. Suppose that for any $a \in S \cup S^*$, the nets $(T_ia)_{i \in I}$ and $(T_i^*a)_{i \in I}$ converge. Then $(T_i)_{i \in I},(T_i^*)_{i \in I}$ converge strictly and $(\lim_{i \in I}T_i)^*=\lim_{i \in I}T_i^*$.
\end{lemma}
\begin{proof}
It is straightforward to check.
\end{proof}

\begin{thm}[{cf. \cite[Corollary~5.3]{KP00}}]\label{O_{G,Lambda} rtimes T^k iso to O_{G,Lambda star Z^k}}
Let $(G,\Lambda)$ be a self-similar $k$-graph. Then $\mathcal{O}_{G,\Lambda} \rtimes_\gamma \mathbb{T}^k$ is isomorphic to $\mathcal{O}_{G,\Lambda \star \mathbb{Z}^k}$.
\end{thm}
\begin{proof}
 Let$(i_A,i_G)$ be the universal covariant homomorphism of $(\mathcal{O}_{G,\Lambda},\mathbb{T}^k,\gamma)$ in $M(\mathcal{O}_{G,\Lambda} \rtimes_\gamma \mathbb{T}^k)$ (cf. \cite[Theorem~2.61]{Williams:Crossedproductsof07}).

For $\mu \in \Lambda,g \in G, \lambda \in \mathbb{T}^k$, define $\pi(s_\mu):=\sum_{z}t_{(\mu,z)} ,\pi(u_g):=\sum_{(v,z)}t_{(v,z)}v_g ,U_\lambda:=\sum_{(v,z)}\lambda^{-z}t_{(v,z)}$ ($\pi(s_\mu),\pi(u_g),U_\lambda$ lie in $M(\mathcal{O}_{G,\Lambda \star \mathbb{Z}^k})$ due to Lemma~\ref{operators in multiplier}). It is straightforward to check that there exists a homomorphism $\pi:\mathcal{O}_{G,\Lambda}^\dagger \to M(\mathcal{O}_{G,\Lambda \star \mathbb{Z}^k})$ which restricts to a nondegenerate homomorphism of $\mathcal{O}_{G,\Lambda}$. Since $\pi(\gamma_\lambda(s_\mu u_g s_\nu^*))=U_\lambda \pi(s_\mu u_g s_\nu^*) U_\lambda^*$ for all $\mu,\nu \in \Lambda,g \in G, \lambda \in \mathbb{T}^k$, we get a nondegenerate covariant homomorphism $(\pi,U)$ of $(\mathcal{O}_{G,\Lambda},\mathbb{T}^k,\gamma)$ in $M(\mathcal{O}_{G,\Lambda \star \mathbb{Z}^k})$. By the universal property of $\mathcal{O}_{G,\Lambda} \rtimes_\gamma \mathbb{T}^k$, there is a nondegenerate homomorphism $\pi \rtimes U:\mathcal{O}_{G,\Lambda} \rtimes_\gamma \mathbb{T}^k \to M(\mathcal{O}_{G,\Lambda \star \mathbb{Z}^k})$ (denote by $\overline{\pi \rtimes U}:M(\mathcal{O}_{G,\Lambda} \rtimes_\gamma \mathbb{T}^k) \to M(\mathcal{O}_{G,\Lambda \star \mathbb{Z}^k})$ the unique extension of $\pi \rtimes U$) such that $\overline{\pi \rtimes U} \circ i_A=\pi,\overline{\pi \rtimes U} \circ i_G=U$. Notice that the image of $\pi \rtimes U$ actually lies in $\mathcal{O}_{G,\Lambda \star \mathbb{Z}^k}$.

Conversely, for $\mu \in \Lambda,z \in \mathbb{Z}^k,g \in G$, define $\rho(t_{(\mu,z)}):=\int \lambda^{d(\mu)+z}i_A(s_\mu)i_G(\lambda) d\lambda,\rho(v_g):=\sum_{v}i_A(s_v u_g)$. It is straightforward to see that there exists a homomorphism $\rho:\mathcal{O}_{G,\Lambda \star \mathbb{Z}^k}^\dagger \to M(\mathcal{O}_{G,\Lambda} \rtimes_\gamma \mathbb{T}^k)$. Observe that $\rho(\mathcal{O}_{G,\Lambda \star \mathbb{Z}^k}) \subset \mathcal{O}_{G,\Lambda} \rtimes_\gamma \mathbb{T}^k$.

It is easy to see that $\rho \circ (\pi \rtimes U)=\id \vert_{ \mathcal{O}_{G,\Lambda} \rtimes_\gamma \mathbb{T}^k}$ and $(\pi \rtimes U) \circ \rho \vert_{\mathcal{O}_{G,\Lambda \star \mathbb{Z}^k}}=\id \vert_{\mathcal{O}_{G,\Lambda \star \mathbb{Z}^k}}$. Hence we are done.
\end{proof}

\begin{cor}\label{O_{G,Lambda star Z^k} ME O_{G,Lambda}}
Let $(G,\Lambda)$ be a self-similar $k$-graph. Then there exists a group homomorphism $\widehat{\gamma}:\mathbb{Z}^k \to \Aut(\mathcal{O}_{G,\Lambda \star \mathbb{Z}^k})$ such that
\begin{enumerate}
\item $\widehat{\gamma}_z(t_{(\mu,w-d(\mu))}v_g t_{(\nu,w-d(\nu))}^*)=t_{(\mu,z+w-d(\mu))}v_g t_{(\nu,z+w-d(\nu))}^*$ for all $\mu,\nu \in \Lambda,g \in G,z,w \in \mathbb{Z}^k$ with $s(\mu)=g \cdot s(\nu)$;
\item $\mathcal{O}_{G,\Lambda \star \mathbb{Z}^k} \rtimes_{\widehat{\gamma}} \mathbb{Z}^k$ is Morita equivalent to $\mathcal{O}_{G,\Lambda}$.
\end{enumerate}
\end{cor}
\begin{proof}
The first statement follows immediately from Theorem~\ref{O_{G,Lambda} rtimes T^k iso to O_{G,Lambda star Z^k}}. The second statement holds due to the Takai duality theorem.
\end{proof}

\section{K-theory of C*-Algebras of Products of Odometers}

\subsection{A conjecture on the K-theory of C*-algebras of products of odometers}

In this subsection, we pose a conjecture on the K-theory of the self-similar $k$-graph C*-algebra of a standard product of odometers.

\begin{prop}\label{O_Z,Lambda star Zk is AT}
Let $(\mathbb{Z},\Lambda)$ be a product of odometers. For $n \geq 1$, define
\begin{align*}
B_n:= \begin{cases}
    \overline{\spn}\{t_{(\mu,n\mathsf{1}_k-d(\mu))}v_g t_{(\nu,n\mathsf{1}_k-d(\nu))}^*\}, &\text{ if } k \neq \infty\\
    \overline{\spn}\{t_{(\mu,n\mathsf{1}_n-d(\mu))}v_g t_{(\nu,n\mathsf{1}_n-d(\nu))}^*\}, &\text{ if } k=\infty.
\end{cases}
\end{align*}
\begin{enumerate}
\item For $n \geq 1, B_n$ is a C*-subalgebra of $\mathcal{O}_{\mathbb{Z},\Lambda \star \mathbb{Z}^k}$.
\item For $n \geq 1, B_n \cong C(\mathbb{T}) \otimes K(l^2(\Lambda))$.
\item\label{O_G,Lambda star Z^k =lim A_n} $\{B_n\}_{n=1}^{\infty}$ is an increasing sequence and $\mathcal{O}_{\mathbb{Z},\Lambda \star \mathbb{Z}^k}=\overline{\bigcup_{n \geq 1 }B_n}$.
\item $K_0(\mathcal{O}_{\mathbb{Z},\Lambda \star \mathbb{Z}^k}) \cong \{\frac{z}{\prod_{i=1}^{k}n_i^{p_i} }:z \in \mathbb{Z},p_1,\dots,p_k \in \mathbb{N},\sum_{i=1}^{k}p_i<\infty \},K_1(\mathcal{O}_{\mathbb{Z},\Lambda \star \mathbb{Z}^k}) \cong \mathbb{Z}$.
\item The homomorphism $\widehat{\gamma}$ in Corollary~\ref{O_{G,Lambda star Z^k} ME O_{G,Lambda}} induces an action of $\mathbb{Z}^k$ on $K_0(\mathcal{O}_{\mathbb{Z},\Lambda \star \mathbb{Z}^k})$ such that $z \cdot w /(n_1\cdots n_k)^{n}=w/((n_1\cdots n_k)^{n}n_1^{z_1}\cdots n_k^{z_k})$ for all $z \in \mathbb{Z}^k,w \in \mathbb{Z},n \geq 1$, and induces a trivial action $\mathbb{Z}^k$ of on $K_1(\mathcal{O}_{\mathbb{Z},\Lambda \star \mathbb{Z}^k})$.
\end{enumerate}
\end{prop}
\begin{proof}
The proof of this proposition is similar for both the cases $k<\infty$ and $k=\infty$. So we only prove the proposition for $k<\infty$ and omit the proof of the other case. We assume that $k<\infty$.

For $n \geq 1,\mu,\nu,\alpha,\beta \in \Lambda,g,h \in \mathbb{Z}$, we have
\begin{align*}
&t_{(\mu,n\mathsf{1}_k-d(\mu))}v_g t_{(\nu,n\mathsf{1}_k-d(\nu))}^*t_{(\alpha,n\mathsf{1}_k-d(\alpha))}v_h t_{(\beta,n\mathsf{1}_k-d(\beta))}^*
= \delta_{\nu,\alpha}t_{(\mu,n\mathsf{1}_k-d(\mu))}v_{gh} t_{(\beta,n\mathsf{1}_k-d(\beta))}^*.
\end{align*}
So $B_n$ is a C*-subalgebra of $\mathcal{O}_{\mathbb{Z},\Lambda \star \mathbb{Z}^k}$.

For $n \geq 1$, by Lemma~\ref{operators in multiplier}, $\sum_{\mu \in \Lambda}t_{(\mu,n\mathsf{1}_k-d(\mu))}v_g t_{(\mu,n\mathsf{1}_k-d(\mu))}^*$ is a unitary of $M(B_n)$ for all $g \in \mathbb{Z}$. Denote by $V:\mathbb{T} \to \mathbb{C},\lambda \mapsto \lambda$ the generating unitary of $C(\mathbb{T})$. Then there exists a homomorphism $\varphi:C(\mathbb{T}) \to M(B_n)$ such that $\varphi(V^g)= \sum_{\mu \in \Lambda}t_{(\mu,n\mathsf{1}_k-d(\mu))}v_g t_{(\mu,n\mathsf{1}_k-d(\mu))}^*$ for all $g \in \mathbb{Z}$. Denote by $\{e_{\mu,\nu}\}_{\mu,\nu \in \Lambda}$ the generators of $K(l^2(\Lambda))$. For $\mu,\nu \in \Lambda$, define $E_{\mu,\nu}:=t_{(\mu,n\mathsf{1}_k-d(\mu))}t_{(\nu,n\mathsf{1}_k-d(\nu))}^*$. For $\mu,\nu,\alpha,\beta \in \Lambda$, we have $E_{\mu,\nu}E_{\alpha,\beta}=\delta_{\nu,\alpha}E_{\mu,\beta}$. So there exists a homomorphism $\psi:K(l^2(\Lambda)) \to B_n$ such that $\psi(e_{\mu,\nu})=E_{\mu,\nu}$ for all $\mu,\nu \in \Lambda$. Since the images of $\varphi,\psi$ commute, by \cite[Theorem~6.3.7]{MurphyCalgebra} there exists a surjective homomorphism $h:C(\mathbb{T}) \otimes K(l^2(\Lambda))\to B_n$ such that $h(V^g \otimes e_{\mu,\nu})= t_{(\mu,n\mathsf{1}_k-d(\mu))} v_g t_{(\nu,n\mathsf{1}_k-d(\nu))}^*$ for all $g \in \mathbb{Z},\mu,\nu \in \Lambda$. Notice that there exist two faithful expectations $E:C(\mathbb{T}) \otimes K(l^2(\Lambda)) \to C(\mathbb{T}) \otimes K(l^2(\Lambda)),F:B_n \to B_n$ (the existence of $F$ follows from \cite[Theorem~3.20]{LY19_4}) such that $E(V^g \otimes e_{\mu,\nu})=\delta_{g,0}\delta_{\mu,\nu}1_{C(\mathbb{T})}\otimes e_{\mu,\mu},F(t_{(\mu,n\mathsf{1}_k-d(\mu))} v_g t_{(\nu,n\mathsf{1}_k-d(\nu))}^*)=\delta_{g,0}\delta_{\mu,\nu}t_{(\mu,n\mathsf{1}_k-d(\mu))} t_{(\mu,n\mathsf{1}_k-d(\mu))}^*$. It is easy to check that $h \circ E=F \circ h$ and $h$ is injective on the image of $E$. Hence $h$ is an isomorphism by \cite[Proposition~3.11]{Kat03}. Therefore $B_n \cong C(\mathbb{T}) \otimes K(l^2(\Lambda))$.

For $n \geq 1,\mu,\nu \in \Lambda,g \in \mathbb{Z}$, we have
\begin{align*}
t_{(\mu,n\mathsf{1}_k-d(\mu))}v_g t_{(\nu,n\mathsf{1}_k-d(\nu))}^*&=\sum_{\alpha \in \Lambda^{\mathsf{1}_k}}t_{(\mu,n\mathsf{1}_k-d(\mu))}v_g t_{(\alpha,n\mathsf{1}_k)}t_{(\alpha,n\mathsf{1}_k)}^* t_{(\nu,n\mathsf{1}_k-d(\nu))}^*
\\&=\sum_{\alpha \in \Lambda^{\mathsf{1}_k}}t_{(\mu(g \cdot \alpha),n\mathsf{1}_k-d(\mu))}v_{g \vert_\alpha} t_{(\nu\alpha,n\mathsf{1}_k-d(\nu))}^*
\\&=\sum_{\alpha \in \Lambda^{\mathsf{1}_k}}t_{(\mu(g \cdot \alpha),(n+1)\mathsf{1}_k-d(\mu(g \cdot \alpha)))}v_{g \vert_\alpha} t_{(\nu\alpha,(n+1)\mathsf{1}_k-d(\nu\alpha))}^*.
\end{align*}
So $B_n \subset B_{n+1}$.

For $\mu,\nu \in \Lambda,w \in \mathbb{Z}^k,g \in \mathbb{Z}$, pick up an arbitrary $n \geq 1$ such that $w \leq n \mathsf{1}_k$, we have
\begin{align*}
t_{(\mu,w-d(\mu))}v_g t_{(\nu,w-d(\nu))}^*&=\sum_{\alpha \in \Lambda^{n\mathsf{1}_k-w}}t_{(\mu,w-d(\mu))}v_gt_{(\alpha,w)}t_{(\alpha,w)}^* t_{(\nu,w-d(\nu))}^*
\\&=\sum_{\alpha \in \Lambda^{n\mathsf{1}_k-w}}t_{(\mu(g \cdot \alpha),w-d(\mu))}v_{g \vert_\alpha} t_{(\nu\alpha,w-d(\nu))}^*
\\&=\sum_{\alpha \in \Lambda^{n\mathsf{1}_k-w}}t_{(\mu(g \cdot \alpha),n\mathsf{1}_k-d(\mu(g \cdot \alpha)))}v_{g \vert_\alpha} t_{(\nu\alpha,n\mathsf{1}_k-d(\nu\alpha))}^*.
\end{align*}
Hence $\mathcal{O}_{\mathbb{Z},\Lambda \star \mathbb{Z}^k}=\overline{\bigcup_{n \geq 1 }B_n}$.

For $n \geq 1$, denote by $\iota_n:B_n \to B_{n+1}$ the inclusion map. we compute that $t_{(v,n\mathsf{1}_k)}=\sum_{\mu \in \Lambda^{\mathsf{1}_k}}t_{(\mu,(n+1)\mathsf{1}_k-\mathsf{1}_k)}t_{(\mu,(n+1)\mathsf{1}_k-\mathsf{1}_k)}^*$. So $K_0(\iota_n)(1)=n_1 \cdots n_k$. Hence $K_0(\mathcal{O}_{\mathbb{Z},\Lambda \star \mathbb{Z}^k})\cong\{$ $\frac{z}{\prod_{i=1}^{k}n_i^{p_i} }:z \in \mathbb{Z},p_1,\dots,p_k \in \mathbb{N}\}$. We calculate that $\iota_n(t_{(v,n\mathsf{1}_k)}v_1 t_{(v,n\mathsf{1}_k)}^*)=\sum_{\alpha \in \Lambda^{\mathsf{1}_k}}$ $t_{(1\cdot\alpha,(n+1)\mathsf{1}_k-\mathsf{1}_k)}v_{1 \vert_\alpha} t_{(\alpha,(n+1)\mathsf{1}_k-\mathsf{1}_k)}^*$, which is regarded as a unitary of $M_{\Lambda^{\mathsf{1}_k}} \otimes C(\mathbb{T})$. Consider another unitary $\sum_{\beta \in \Lambda^{\mathsf{1}_k}}t_{(\beta,(n+1)\mathsf{1}_k-\mathsf{1}_k)}v_{-(1 \vert_\beta)} t_{(\beta,(n+1)\mathsf{1}_k-\mathsf{1}_k)}^*$ of $M_{\Lambda^{\mathsf{1}_k}} \otimes C(\mathbb{T})$. We compute that
\begin{align*}
&\sum_{\alpha \in \Lambda^{\mathsf{1}_k}}t_{(1\cdot\alpha,(n+1)\mathsf{1}_k-\mathsf{1}_k)}v_{1 \vert_\alpha} t_{(\alpha,(n+1)\mathsf{1}_k-\mathsf{1}_k)}^* \sum_{\beta \in \Lambda^{\mathsf{1}_k}}t_{(\beta,(n+1)\mathsf{1}_k-\mathsf{1}_k)}v_{-(1 \vert_\beta)} t_{(\beta,(n+1)\mathsf{1}_k-\mathsf{1}_k)}^*
\\&=\sum_{\alpha \in \Lambda^{\mathsf{1}_k}}t_{(1\cdot\alpha,(n+1)\mathsf{1}_k-\mathsf{1}_k)} t_{(\alpha,(n+1)\mathsf{1}_k-\mathsf{1}_k)}^*
\\&\sim_h \sum_{\alpha \in \Lambda^{\mathsf{1}_k}}t_{(\alpha,(n+1)\mathsf{1}_k-\mathsf{1}_k)} t_{(\alpha,(n+1)\mathsf{1}_k-\mathsf{1}_k)}^*.
\end{align*}
So $[\sum_{\alpha \in \Lambda^{\mathsf{1}_k}}t_{(1\cdot\alpha,(n+1)\mathsf{1}_k-\mathsf{1}_k)}v_{1 \vert_\alpha} t_{(\alpha,(n+1)\mathsf{1}_k-\mathsf{1}_k)}^*] =[\sum_{\beta \in \Lambda^{\mathsf{1}_k}}t_{(\beta,(n+1)\mathsf{1}_k-\mathsf{1}_k)}v_{1 \vert_\beta}t_{(\beta,(n+1)\mathsf{1}_k-\mathsf{1}_k)}^*]$ in $K_1(M_{\Lambda^{\mathsf{1}_k}} \otimes C(\mathbb{T}))$. Since $1 \vert_{x_{\fs_1}^1 \cdots x_{\fs_k}^k}=\delta_{\fs_1,n_1-1}\cdots\delta_{\fs_k,n_k-1}1, K_1(\iota_n)=\id$. Hence $K_1(\mathcal{O}_{\mathbb{Z},\Lambda \star \mathbb{Z}^k})$ $\cong \mathbb{Z}$.

Since $\widehat{\gamma}$ is a group homomorphism from $\mathbb{Z}^k$ to $\Aut(\mathcal{O}_{G,\Lambda \star \mathbb{Z}^k}), \widehat{\gamma}$ induces actions of $\mathbb{Z}^k$ on $K_0(\mathcal{O}_{\mathbb{Z},\Lambda \star \mathbb{Z}^k})$ and $K_1(\mathcal{O}_{\mathbb{Z},\Lambda \star \mathbb{Z}^k})$. For $z \in \mathbb{Z}^k,n \geq 1$, there exists $m \geq 1$ such that $z \leq m \mathsf{1}_k$. Then $\widehat{\gamma}_{z}(t_{(v,n\mathsf{1}_k)})=t_{(v,n\mathsf{1}_k+z)}=\sum_{\alpha \in \Lambda^{m \mathsf{1}_k-z}}t_{(\alpha,(m+n)\mathsf{1}_k-d(\alpha))}t_{(\alpha,(m+n)\mathsf{1}_k-d(\alpha))}^*$. So $K_0(\widehat{\gamma}_{z})(1/(n_1\cdots n_k)^n)=1/((n_1\cdots n_k)^n n_1^{z_1}\cdots n_k^{z_k})$. Moreover, we calculate that
\begin{align*}
\widehat{\gamma}_{z}(v_1 t_{(v,\mathsf{1}_k)})&=v_1  t_{(v,\mathsf{1}_k+z)}
\\&=v_1 \sum_{\alpha \in \Lambda^{m \mathsf{1}_k-z}}t_{(\alpha,(m+1)\mathsf{1}_k-d(\alpha))}t_{(\alpha,(m+1)\mathsf{1}_k-d(\alpha))}^*
\\&=\sum_{\alpha \in \Lambda^{m \mathsf{1}_k-z}}t_{(1 \cdot \alpha,(m+1)\mathsf{1}_k-d(\alpha))} v_{1 \vert_\alpha}t_{(\alpha,(m+1)\mathsf{1}_k-d(\alpha))}^*
\end{align*}
So $K_1(\widehat{\gamma}_{z})$ is the identity map.
\end{proof}

\begin{conj}\label{decompose K(O_G,Lambda)}
Let $(\mathbb{Z},\Lambda)$ be a product of odometers with $k \geq 2$. Denote by $e:=(\delta_{1,i}+\mathbb{Z}/g\Lambda\mathbb{Z})_{1 \leq i \leq k} \in  (\mathbb{Z}/g\Lambda\mathbb{Z})^{2^{k-2}}$. Then
\[
(K_0(\mathcal{O}_{\mathbb{Z},\Lambda}),[1_{\mathcal{O}_{\mathbb{Z},\Lambda}}], K_1(\mathcal{O}_{\mathbb{Z},\Lambda}) ) \cong (\mathbb{Z}^{2^{k-1}}\oplus (\mathbb{Z}/g\Lambda\mathbb{Z})^{2^{k-2}},(0,e),\mathbb{Z}^{2^{k-1}}\oplus (\mathbb{Z}/g\Lambda\mathbb{Z})^{2^{k-2}}).
\]
\end{conj}

\begin{rmk}
We first notice that the C*-dynamical systems $(\mathcal{O}_{\mathbb{Z},\Lambda \star \mathbb{Z}^k},\mathbb{Z},\widehat{\gamma})$ and $(\mathcal{O}_{\mathbb{Z},\Lambda \star \mathbb{Z}^k},\mathbb{Z}$, $\widehat{\gamma}^{-1})$ yield the same crossed product C*-algebra. By Corollary~\ref{O_{G,Lambda star Z^k} ME O_{G,Lambda}}, $K_*(\mathcal{O}_{\mathbb{Z},\Lambda})\cong K_*(\mathcal{O}_{\mathbb{Z},\Lambda \star \mathbb{Z}^k}$ $\rtimes_{\widehat{\gamma}} \mathbb{Z}^k) \cong K_*(\mathcal{O}_{\mathbb{Z},\Lambda \star \mathbb{Z}^k} \rtimes_{\widehat{\gamma}^{-1}} \mathbb{Z}^k)$. For $k=1$, the K-theory of the C*-algebra $\mathcal{O}_{\mathbb{Z},\Lambda}$ was studied in numerous papers such as \cite{MR3089668, MR1896827, Kat08_1, MR2873845}.

Now we assume that $k \geq 2$. It is easier to discuss the K-theory of $\mathcal{O}_{\mathbb{Z},\Lambda \star \mathbb{Z}^k} \rtimes_{\widehat{\gamma}^{-1}} \mathbb{Z}^k$ and our discussion is based on the combination of previous profound work (\cite[Corollary~2.5]{Bar15}, \cite[6.10]{MR918241}, \cite[Theorem~2]{MR2505326}, \cite{Schochet:PJM81}) give rise to a cohomology spectral sequence (see \cite{MR1269324}) $\{E_r^{p,q}\}_{r \geq 1,p,q \in \mathbb{Z}}$, where for $p,q \in \mathbb{Z}$, we have $E_1^{p,q}=K_q(\mathcal{O}_{\mathbb{Z},\Lambda \star \mathbb{Z}^k} ) \otimes_{\mathbb{Z}} \bigwedge^p(\mathbb{Z}^k)$, the differential map is $d_1^{p,q}:E_1^{p,q} \to E_1^{p+1,q}, g \otimes e \mapsto \sum_{i=1}^{k}(e_i \cdot g-g) \otimes (e \wedge e_i)$ for all $g \in K_q(\mathcal{O}_{\mathbb{Z},\Lambda \star \mathbb{Z}^k} ),e \in \bigwedge^p(\mathbb{Z}^k)$. When $q$ is even, following the algorithm from \cite[Proposition~6.12]{MR3784245}, $E_2^{p,q}=\begin{cases}
    (\mathbb{Z}/g_\Lambda \mathbb{Z})^{C_{k-1}^{p-1}} &\text{ if }1\leq p \leq k\\
    0 &\text{ if } p<1  \text{ or } p >k
\end{cases}$. On the other hand, when $q$ is odd, since the action on $K_1(\mathcal{O}_{\mathbb{Z},\Lambda \star \mathbb{Z}^k} )$ is trivial due to Proposition~\ref{O_Z,Lambda star Zk is AT}, we can easily deduce that the connecting maps $d_1^{p,q}$ are all zero, hence $E_2^{p,q}=\begin{cases}
    \mathbb{Z}^{C_{k}^{p}} &\text{ if }0\leq p \leq k\\
    0 &\text{ if } p<0  \text{ or } p >k
\end{cases}$. The determination of the $k+1$-th page of the spectral sequence is not easy, even though one can achieve this, there is another extension obstruction to characterize the K-theory of $\mathcal{O}_{\mathbb{Z},\Lambda \star \mathbb{Z}^k} \rtimes_{\widehat{\gamma}^{-1}} \mathbb{Z}^k$.
\end{rmk}

\subsection{Generalized $\mathcal{Q}_S$ and a conjecture of Barlak, Omland, and Stammeier}

$\linebreak$

For any nonempty subset $S \subset \mathbb{N}^\times \setminus \{1\}$ consisting of mutually coprime numbers, Barlak, Omland, and Stammeier in \cite{MR3784245} defined a unital C*-algebra $\mathcal{Q}_S$ and they made a conjecture about the K-theory of this C*-algebra.

\begin{conj}[{\cite[Conjecture~6.5]{MR3784245}}]\label{BOS conj}
\[
(K_0(\mathcal{Q}_S),[1_{\mathcal{Q}_S}], K_1(\mathcal{Q}_S) ) \cong (\mathbb{Z}^{2^{\vert S\vert-1}}\oplus (\mathbb{Z}/g_S \mathbb{Z})^{2^{\vert S\vert-2}},(0,e),\mathbb{Z}^{2^{\vert S\vert-1}}\oplus (\mathbb{Z}/g_S\mathbb{Z})^{2^{\vert S\vert-2}}),
\]
where $g_S=\mathrm{gcd}\{n-1:n \in S\},e=(\delta_{1,i}+\mathbb{Z}/g\Lambda\mathbb{Z})_{1 \leq i \leq k} \in  (\mathbb{Z}/g\Lambda\mathbb{Z})^{2^{k-2}}$.
\end{conj}

In the final subsection, we first generalize the construction of $\mathcal{Q}_S$ to an arbitrary nonempty subset of $\mathbb{N}^\times\setminus \{1\}$, and we show that $\mathcal{Q}_S$ is indeed a self-similar $\vert S\vert$-graph C*-algebra of a standard product of odometers, finally we connect Conjecture~\ref{decompose K(O_G,Lambda)} with Conjecture~\ref{BOS conj}.

\begin{defn}[{cf. \cite[Definition~2.1]{MR3784245}}]\label{define Q_S}
Let $S$ be a nonempty subset of $\mathbb{N}^\times \setminus \{1\}$. Define $\mathcal{Q}_S$ to be the universal unital C*-algebra generated by a family of isometries $\{s_n\}_{n \in S}$ and a unitary $u$ satisfying for any $n,m \in S$,
\begin{enumerate}
\item\label{s_n s_m=s_m s_n} $s_n s_m=s_m s_n$;
\item\label{s_n u=u^n s_n} $s_n u=u^n s_n$;
\item\label{sum u^i s_n (u^i s_n)*} $\sum_{i=0}^{n-1}u^i s_n s_n^*u^{-i}=1_{\mathcal{Q}_S}$.
\end{enumerate}
\end{defn}

\begin{rmk}
It is very natural to extend Conjecture~\ref{BOS conj} to the case that $S$ is any nonempty subset of $\mathbb{N}^\times \setminus \{1\}$.
\end{rmk}

\begin{lemma}\label{s_n* u^z s_n=0}
Let $S$ be a nonempty subset of $\mathbb{N}^\times \setminus \{1\}$, and let $B$ be a unital C*-algebra generated by a family of isometries $\{S_n\}_{n \in S}$ and a unitary $U$ satisfying Conditions~(\ref{s_n u=u^n s_n}), (\ref{sum u^i s_n (u^i s_n)*}) of Definition~\ref{define Q_S}. Then for $n \in S,z \in \mathbb{Z} \setminus n \mathbb{Z}$, we have $S_n^* U^z S_n=0$.
\end{lemma}
\begin{proof}
We may assume that $z >0$. Write $z=wn+l$, for some $w \geq 0,1 \leq l \leq n-1$. We calculate that
\begin{align*}
(S_n^* U^z S_n)(S_n^* U^z S_n)^*&=S_n^* U^l U^{wn} S_n S_n^* U^{-wn} U^{-l} S_n
\\&=S_n^* U^l (U^{wn} S_n) (U^{wn} S_n)^* U^{-l} S_n
\\&=S_n^* U^l ( S_n U^w) (S_n U^w)^* U^{-l} S_n
\\&=S_n^* U^l  S_n S_n^* U^{-l} S_n
\\&=S_n^* U^l  (1_B-\sum_{i=1}^{n-1}U^i S_n S_n^* U^{-i}) U^{-l} S_n
\\&=1_B-S_n^* U^n S_n S_n^* U^{-n} S_n-\sum_{1 \leq i \leq n-1,i \neq n-l}S_n^* U^{l+i} S_n S_n^* U^{-l-i} S_n
\\&=-\sum_{1 \leq i \leq n-1,i \neq n-l}S_n^* U^{l+i} S_n S_n^* U^{-l-i} S_n.
\end{align*}
So $(S_n^* U^z S_n)(S_n^* U^z S_n)^*=0$. Hence $S_n^* U^z S_n=0$.
\end{proof}

\begin{prop}\label{s_n*s_m}
Let $S$ be a nonempty subset of $\mathbb{N}^\times \setminus \{1\}$, and let $B$ be a unital C*-algebra generated by a family of isometries $\{S_n\}_{n \in S}$ and a unitary $U$ satisfying Conditions~(\ref{s_n u=u^n s_n}), (\ref{sum u^i s_n (u^i s_n)*}) of Definition~\ref{define Q_S}. Then $\{S_n\}_{n \in S}$ and $U$ satisfy Condition~(\ref{s_n s_m=s_m s_n}) of Definition~\ref{define Q_S} if and only if for any $n,m \in S,S_n^*S_m=\sum_{0 \leq l \leq nm-1,l \in n\mathbb{Z} \cap m \mathbb{Z}}U^{l/n} S_m S_n^* U^{-l/m}$.
\end{prop}
\begin{proof}
First of all, suppose that $\{S_n\}_{n \in S}$ and $U$ satisfy Condition~(\ref{s_n s_m=s_m s_n}) of Definition~\ref{define Q_S}. For $n,m \in S$, we compute that
\begin{align*}
S_n^*S_m&=\sum_{i=0}^{n-1}S_n^* U^i S_n S_n^*U^{-i} S_m
\\&=\sum_{i=0}^{n-1}\sum_{j=0}^{m-1}S_n^* U^i S_n U^j S_m S_m^* U^{-j} S_n^* U^{-i} S_m
\\&=\sum_{i=0}^{n-1}\sum_{j=0}^{m-1}S_n^* U^{i+nj} S_n S_m S_m^* S_n^* U^{-i-nj} S_m
\\&=\sum_{l=0}^{nm-1}S_n^* U^{l} S_n S_m S_m^* S_n^*U^{-l}S_m
\\&=\sum_{l=0}^{nm-1}(S_n^* U^{l} S_n) S_m S_n^* (S_m^* U^{-l}S_m)
\\&=\sum_{0 \leq l \leq nm-1,l \in n\mathbb{Z} \cap m \mathbb{Z}}U^{l/n} S_m S_n^* U^{-l/m} \text{ (by Lemma~\ref{s_n* u^z s_n=0})}.
\end{align*}

Conversely, suppose that for any $n,m \in S,S_n^*S_m=\sum_{0 \leq l \leq nm-1,l \in n\mathbb{Z} \cap m \mathbb{Z}}U^{l/n} S_m S_n^* U^{-l/m}$. For $n,m \in S$, we calculate that
\begin{align*}
S_m^* S_n^* S_m S_n&=\sum_{0 \leq l \leq nm-1,l \in n\mathbb{Z} \cap m \mathbb{Z}}(S_m^*U^{l/n} S_m) (S_n^* U^{-l/m}S_n)
\\&=1_B+\sum_{1 \leq l \leq nm-1,l \in n\mathbb{Z} \cap m \mathbb{Z}}(S_m^*U^{l/n} S_m) (S_n^* U^{-l/m}S_n)
\\&=1_B \text{ (by Lemma~\ref{s_n* u^z s_n=0})}.
\end{align*}
So
\[
(S_n S_m-S_m S_n)^*(S_n S_m-S_m S_n)=2 \cdot 1_B-S_m^* S_n^* S_m S_n-S_n^* S_m^* S_n S_m=0.
\]
Hence $S_n S_m=S_m S_n$.
\end{proof}

\begin{rmk}\label{simplification of O_G,Lambda for prod of odo}
Let $(\mathbb{Z},\Lambda)$ be a standard product of odometers. Then $\mathcal{O}_{\mathbb{Z},\Lambda}$ is a universal unital C*-algebra generated by a family of isometries $\{s_{x_{\fs}^{i}}\}_{1 \leq i \leq k,0 \leq \fs \leq n_i-1}$ and a unitary $u$ satisfying that
\begin{enumerate}
\item\label{CK relation}$\sum_{\fs=0}^{n_i-1}s_{x_{\fs}^{i}}s_{x_{\fs}^{i}}^*=1_{\mathcal{O}_{\mathbb{Z},\Lambda}}$ for all $1 \leq i \leq k$;
\item $s_{x_{\fs}^{i}}s_{x_{\ft}^{j}}=s_{x_{\ft'}^{j}}s_{x_{\fs'}^{i}}$ for all $1\leq i<j \leq k,0\le \fs,\fs'\le n_i-1,0\le \ft,\ft'\le n_j-1$ with $\fs+\ft n_i=\ft'+\fs' n_j$;
\item\label{fg_s^i} $u s_{x_{\fs}^{i}}= \begin{cases}
    s_{x_{\fs+1}^{i}} &\text{ if } 0 \leq \fs <n_i-1 \\
    s_{x_{0}^{i}} u &\text{ if } \fs=n_i-1
\end{cases}
\text{, for all } 1\le i\le k,0\le \fs\le n_i-1$.
\end{enumerate}
\end{rmk}

\begin{thm}\label{Q_S iso to O_G,Lambda}
Let $S$ be a nonempty subset of $\mathbb{N}^\times \setminus \{1\}$. We enumerate $S=\{1<n_1< \dots <n_k\}$. Denote by $(\mathbb{Z},\Lambda_S)$ the standard product of odometers such that $\vert\Lambda_S^{e_i}\vert=n_i$ for all $1 \leq i \leq k$. Then $\mathcal{Q}_S \cong \mathcal{O}_{\mathbb{Z},\Lambda_S}$.
\end{thm}
\begin{proof}
Denote by $\{s_{n_i}\}_{1 \leq i \leq k}$ and $u$ the generators of $\mathcal{Q}_S$, and denote by $\{t_{x_{\fs}^{i}}:1\le i\le k,0\le \fs\le n_i-1\}$ and $v$ the generators of $\mathcal{O}_{\mathbb{Z},\Lambda_S}$ as discussed in Remark~\ref{simplification of O_G,Lambda for prod of odo}.

Define $V:=u, T_{x_{\fs}^i}:=u^{\fs} s_{n_i}$ for all $1\le i\le k,0\le \fs\le n_i-1$. Then $\{T_{x_{\fs}^{i}}:1\le i\le k,0\le \fs\le n_i-1\}$ and $V$ satisfy Conditions~(\ref{CK relation})--(\ref{fg_s^i}) of Remark~\ref{simplification of O_G,Lambda for prod of odo}. So there exists a homomorphism $\pi:\mathcal{O}_{\mathbb{Z},\Lambda_S} \to \mathcal{Q}_S$ such that $\pi(t_{x_{\fs}^i})=u^{\fs}s_{n_i},\pi(v)=u$ for all $1\le i\le k,0\le \fs\le n_i-1$.

Conversely define $U:=v,S_{n_i}:=t_{x_0^i}$ for all $1\le i\le k$. Then $\{S_{n_i}:1 \leq i \leq k\}$ and $U$ satisfy Conditions~(\ref{s_n s_m=s_m s_n})--(\ref{sum u^i s_n (u^i s_n)*}) of Definition~\ref{define Q_S}. So there exists a homomorphism $\rho:\mathcal{Q}_S \to \mathcal{O}_{\mathbb{Z},\Lambda_S}$ such that $\rho(u)=v,\rho(s_{n_i})=s_{x_0^i}$ for all $1 \leq i \leq k$.

Since $\rho\circ\pi=\id,\pi\circ\rho=\id$, we deduce that $\mathcal{Q}_S \cong \mathcal{O}_{\mathbb{Z},\Lambda_S}$.
\end{proof}

\begin{rmk}
For any nonempty subset $S \subset \mathbb{N}^\times \setminus \{1\}$, by the above theorem we can easily deduce that Conjecture~\ref{BOS conj} is contained in Conjecture~\ref{decompose K(O_G,Lambda)}.
\end{rmk}


\end{document}